\documentclass[11pt]{article}

\usepackage{amssymb,amsmath,amsfonts,amsthm}
\usepackage{latexsym}
\usepackage{graphics}
\usepackage{indentfirst}

\newtheorem{innercustomthm}{Theorem}
\newenvironment{customthm}[1]
  {\renewcommand\theinnercustomthm{#1}\innercustomthm}
  {\endinnercustomthm}
\newtheorem{innercustomlem}{Lemma}
\newenvironment{customlem}[1]
  {\renewcommand\theinnercustomlem{#1}\innercustomlem}
  {\endinnercustomlem}

\newtheorem{innercustompro}{Proposition}
\newenvironment{custompro}[1]
  {\renewcommand\theinnercustompro{#1}\innercustompro}
  {\endinnercustompro}

\newtheorem*{thm*}{Theorem}
\newtheorem{thm}{Theorem}
\newtheorem{lem}[thm]{Lemma}
\newtheorem{pro}[thm]{Proposition}
\newtheorem{obs}[thm]{Observation}

\newtheorem{conj}[thm]{Conjecture}

\newcommand{\N}{\mathbb{N}}

\newcommand{\E}{\mathbb{E}}

\begin{document}

\title{On Equitable List Arboricity of Graphs}

\author{Hemanshu Kaul\footnote{Department of Applied Mathematics, Illinois Institute of Technology, Chicago, IL 60616. E-mail: {\tt kaul@iit.edu}} \\
Jeffrey A. Mudrock\footnote{Department of Mathematics, College of Lake County, Grayslake, IL 60030. E-mail: {\tt jmudrock@clcillinois.edu}} \\
Michael J. Pelsmajer\footnote{Department of Applied Mathematics, Illinois Institute of Technology, Chicago, IL 60616. E-mail: {\tt pelsmajer@iit.edu}} }

\maketitle

\begin{abstract}
	Equitable list arboricity, introduced by Zhang in 2016, generalizes the notion of equitable list coloring by requiring the subgraph induced by each color class to be acyclic (instead of edgeless) in addition to the usual upper bound on the size of each color class.  Graph $G$ is \emph{equitably $k$-list arborable} if an equitable, arborable list coloring of $G$ exists for every list assignment for $G$ that associates with each vertex in $G$ a list of $k$ available colors. Zhang conjectured that any graph $G$ is equitably $k$-list arborable for each $k$ satisfying $k \geq \lceil (1+\Delta(G))/2 \rceil$. We verify this conjecture for powers of cycles by applying a new lemma which is a general tool for extending partial equitable, arborable list colorings. We also propose a stronger version of Zhang's Conjecture for certain connected graphs: any connected graph $G$ is equitably $k$-list arborable for each $k$ satisfying $k \geq \lceil \Delta(G)/2 \rceil$ provided $G$ is neither a cycle nor a complete graph of odd order.  We verify this stronger version of Zhang's Conjecture for powers of paths, 2-degenerate graphs, and certain other graphs.  We also show that if $G$ is equitably $k$-list arborable it does not necessarily follow that $G$ is equitably $(k+1)$-list arborable which addresses a question of Drgas-Burchardt, Furma{\' n}czyk, and Sidorowicz (2018).
	

\medskip



\end{abstract}

\section{Introduction}\label{intro}

In this paper all graphs are nonempty, finite, simple graphs unless otherwise noted.  Generally speaking we follow West~\cite{W01} for terminology and notation.  The set of natural numbers is $\N = \{1,2,3, \ldots \}$.  For $m \in \N$, we write $[m]$ for the set $\{1, \ldots, m \}$.  If $G$ is a graph and $S, U \subseteq V(G)$, we use $G[S]$ for the subgraph of $G$ induced by $S$, and we use $E_G(S, U)$ for the subset of $E(G)$ with at least one endpoint in $S$ and at least one endpoint in $U$.  If an edge in $E(G)$ connects the vertices $u$ and $v$, the edge can be represented by $uv$ or $vu$.  We use $\alpha(G)$ and $\omega(G)$ for the size of the largest independent set and the size of the largest clique in $G$ respectively.  For $v \in V(G)$, we write $d_G(v)$ for the degree of vertex $v$ in the graph $G$, and we use $\Delta(G)$ for the maximum degree of a vertex in $G$.  We say $G$ is \emph{k-degenerate} when every subgraph of $G$ has a vertex of degree at most $k$.  We write $N_G(v)$ for the neighborhood of vertex $v$ in the graph $G$.   Also, $G^k$ denotes the $k^{th}$ power of graph $G$ (i.e., $G^k$ has the same vertex set as $G$ and edges between any two vertices within distance $k$ in $G$).  When $G_1$ and $G_2$ are vertex disjoint graphs we use $G_1+G_2$ to denote the disjoint union of $G_1$ and $G_2$.   If $P$ is a path, $V(P)= \{v_1, \ldots, v_n \}$, and two vertices are adjacent in $P$ if and only if they appear consecutively in the ordering: $v_1, \ldots, v_n$, then we say the vertices are written \emph{in order} when we write $v_1, \ldots, v_n$.  If $C$ is a cycle, $V(C)= \{v_1, \ldots, v_n \}$, and $E(C) = \{\{v_1,v_2 \}, \{v_2, v_3 \}, \ldots, \{v_{n-1}, v_n \}, \{v_n, v_1 \} \}$, then we say the vertices are written in \emph{cyclic order} when we write $v_1, \ldots, v_n$.

In classical vertex coloring, a \emph{proper $k$-coloring} of graph $G$ is a function $f: V(G) \rightarrow A$ where $A$ is a set of colors of size $k$ and $f(u) \neq f(v)$ whenever $uv \in E(G)$.  For each $c \in A$ we say that $f^{-1}(c)$ is the \emph{color class} of $f$ corresponding to $c$.  Clearly, $f^{-1}(c)$ is an independent set in $G$.  In this paper we study a variant of classical vertex coloring called equitable list arboricity which was introduced by Zhang~\cite{Z16}.  This notion combines the notions of equitable coloring, list coloring, and vertex arboricity. So, we begin by briefly reviewing these notions.

\subsection{Equitable Coloring and Equitable Choosability}

Equitable coloring is a variation on classical vertex coloring that was formally introduced by Meyer in the 1973~\cite{M73} (though the study of equitable coloring began in 1964 with a conjecture of Erd\H{o}s~\cite{E64}).  An \emph{equitable $k$-coloring} of a graph $G$ is a proper $k$-coloring of $G$ such that the sizes of the color classes differ by at most one (where a proper $k$-coloring has exactly $k$ color classes).  In an equitable $k$-coloring, it is easy to verify that the color classes associated with the coloring are each of size $\lceil |V(G)|/k \rceil$ or $\lfloor |V(G)|/k \rfloor$.  We say that a graph $G$ is \emph{equitably $k$-colorable} if there exists an equitable $k$-coloring of $G$.

Unlike classical vertex coloring, increasing the number of colors can make equitable coloring more difficult.  Indeed for any $m \in \N$, $K_{2m+1, 2m+1}$ is equitably $2m$-colorable, but it is not equitably $(2m+1)$-colorable.  In 1970, Hajn\'{a}l and Szemer\'{e}di~\cite{HS70} proved the 1964 conjecture of Erd\H{o}s: every graph $G$ has an equitable $k$-coloring when $k \geq \Delta(G)+1$.  In 1994, Chen, Lih, and Wu~\cite{CL94} conjectured that the result of  Hajn\'{a}l and Szemer\'{e}di can be improved by 1 for most connected graphs by characterizing the extremal graphs as: $K_m$, $C_{2m+1}$, and $K_{2m+1,2m+1}$. Their conjecture is still open and is known as the Equitable $\Delta$-Coloring Conjecture. It has received considerable attention in the literature (see e.g.,~\cite{CL94, DW18, DZ18, L98, LW96, YZ97}).

List coloring is yet another variation on classical vertex coloring, and it was introduced independently by Vizing~\cite{V76} and Erd\H{o}s, Rubin, and Taylor~\cite{ET79} in the 1970s.  For list coloring, we associate with graph $G$ a \emph{list assignment}, $L$, that assigns to each vertex $v \in V(G)$ a list, $L(v)$, of available colors. Graph $G$ is said to be \emph{$L$-colorable} if there exists a proper coloring $f$ of $G$ such that $f(v) \in L(v)$ for each $v \in V(G)$ (we refer to $f$ as a \emph{proper $L$-coloring} of $G$).  A list assignment $L$ is called a \emph{k-assignment} for $G$ if $|L(v)|=k$ for each $v \in V(G)$.  We say $G$ is \emph{k-choosable} if $G$ is $L$-colorable whenever $L$ is a $k$-assignment for $G$.  Notice that unlike equitable $k$-colorability, when a graph is $k$-choosable it immediately follows that it is also $(k+1)$-choosable.

In 2003, the third author along with Kostochka and West introduced a notion combining equitable coloring and list coloring known as equitable choosability~\cite{KP03}.  Suppose that $L$ is a $k$-assignment for graph $G$.  An \emph{equitable $L$-coloring} of $G$ is a proper $L$-coloring $f$ of $G$ such that $f$ uses no color more than $\lceil |V(G)|/k \rceil$ times.  When an equitable $L$-coloring of $G$ exists, we say that $G$ is \emph{equitably $L$-colorable}.  Graph $G$ is \emph{equitably $k$-choosable} if $G$ is equitably $L$-colorable whenever $L$ is a $k$-assignment for $G$.  It is important to note that, similar to equitable coloring, making the lists larger may make equitable list coloring more difficult.  Indeed, $K_{1,9}$ is equitably 4-choosable, but it is not equitably 5-choosable.  Also, equitable $k$-choosability does not imply equitable $k$-colorability unless $k=2$ (see~\cite{MC19}).

It is conjectured in~\cite{KP03} that the Hajn\'{a}l-Szemer\'{e}di Theorem and the Equitable $\Delta$-Coloring Conjecture hold in the context of equitable choosability. Both of these conjectures have received quite a bit of attention in the literature (see e.g.,~\cite{KM18, KK13, LB09, M18, ZW11, ZB08, ZB10, ZB15}). We formally state this second conjecture.

\begin{conj}[\cite{KP03}] \label{conj: KPW2}
	A connected graph $G$ is equitably $k$-choosable for each $k \geq \Delta(G)$ if it is different from $K_m$, $C_{2m+1}$, and $K_{2m+1,2m+1}$.
\end{conj}



\subsection{List Vertex Arboricity and Equitable Vertex Arboricity}

The study of vertex arboricity (also called point arboricity) began in the 1960s~\cite{B64, CK69}.   Graph $G$ is \emph{vertex $k$-arborable} if there is a $k$-coloring (not necessarily proper) $f$ of $G$ such that for each color class $f^{-1}(c)$, $G[f^{-1}(c)]$ is acyclic (we call $f$ an \emph{arborable, vertex $k$-coloring of $G$})\footnote{Such a coloring is also referred to as a \emph{tree-$k$-coloring} in the literature.}.  In 2000, Borodin, Kostochka, and Toft~\cite{BK00} introduced a list version of vertex arboricity. If $L$ is a list assignment for $G$, we say that $G$ is \emph{$L$-arborable} if there is an $L$-coloring (not necessarily proper) $f$ of $G$ such that for each color class $f^{-1}(c)$, $G[f^{-1}(c)]$ is acyclic (we call $f$ an \emph{arborable $L$-coloring of $G$}). Graph $G$ is \emph{$k$-list arborable} if for any $k$-assignment $L$ for $G$, $G$ is $L$-arborable.\footnote{We use terms like \emph{$L$-arborable} and \emph{$k$-list arborable} rather than terms like \emph{vertex $L$-arborable} and \emph{vertex $k$-list arborable} since it is understood in this paper that all list assignments associate lists with vertices of a graph.}  In the study of list vertex arboricity, the following analogue of Brooks' Theorem is a well-known result.

\begin{thm} [\cite{BK00, BD95}] \label{thm: upperarbor}
Suppose $G$ is a connected non-complete graph with $\Delta(G) \geq 3$.  Then $G$ is $k$-list arborable whenever $k \geq \Delta(G)/2$.
\end{thm}

In 2013, Wu, Zhang, and Li~\cite{WZ13} introduced an equitable version of vertex arboricity.  Specifically, a graph $G$ is \emph{equitably vertex $k$-arborable} if there exists an arborable, vertex $k$-coloring $f$ of $G$ such that each pair of the $k$ color classes associated with $f$ differ in size by at most one (we call $f$ an \emph{equitable, arborable, vertex $k$-coloring of $G$})\footnote{Such a coloring is also referred to as an \emph{equitable tree-k-coloring} in the literature.}.  As in the case of equitable coloring, a graph that is equitably vertex $k$-arborable need not be equitably vertex $(k+1)$-arborable.  Indeed, $K_{9,9}$ is equitably vertex 2-arborable, but it is not equitably vertex 3-arborable.  The following conjecture of Wu, Zhang, and Li is well-known and has received some attention in the literature (see e.g.,~\cite{DD18, NZ20, WZ13, Z162, ZN20, ZN19, ZW14}).

\begin{conj} [Equitable Vertex Arboricity Conjecture~\cite{WZ13}] \label{conj: varbor}
Graph $G$ is equitably vertex $k$-arborable whenever $k \geq \lceil (\Delta(G)+1)/2 \rceil$.
\end{conj}

\subsection{Equitable List Arboricity}

In 2016, Zhang~\cite{Z16} introduced a list analogue of equitable vertex arboricity which is the focus of this paper.  Graph $G$ is \emph{equitably $k$-list arborable} if for every $k$-assignment $L$ for $G$ there is an arborable $L$-coloring $f$ of $G$ such that each color class of $f$ is of size at most $\lceil |V(G)|/k \rceil$ (we call $f$ an \emph{equitable, arborable $L$-coloring of $G$}).  Zhang made the following conjecture which is a list analogue of the Equitable Vertex Arboricity Conjecture.

\begin{conj} [\cite{Z16}] \label{conj: HSanalogue}
Any graph $G$ is equitably $k$-list arborable for each $k$ satisfying $k \geq \lceil (\Delta(G)+1)/2 \rceil$.
\end{conj}

Complete graphs demonstrate the tightness of the bound in Conjecture~\ref{conj: HSanalogue} since it is easy to prove that $K_n$ is equitably $k$-list arborable if and only if $k \geq \lceil n/2 \rceil = \lceil (1+\Delta(K_n))/2 \rceil$.  Furthermore, all $n$-vertex graphs are equitably $k$-list arborable when $k \geq \lceil n/2 \rceil$ because if $G$ is equitably $k$-list arborable, then any spanning subgraph of $G$ must also be equitably $k$-list arborable.  Conjecture~\ref{conj: HSanalogue} has been verified for 2-degenerate graphs, 3-degenerate claw-free graphs with maximum degree at least~4, and planar graphs with maximum degree at least 8~\cite{Z16}.  It has also been verified for $d$-dimensional grids when $d \in \{2,3,4\}$~\cite{DD18}.

With the Equitable $\Delta$-Coloring Conjecture, Conjecture~\ref{conj: KPW2}, and Theorem~\ref{thm: upperarbor} in mind, we conjecture that the bound in Conjecture~\ref{conj: HSanalogue} can be improved.

\begin{conj} \label{conj: stronger}
Any connected graph $G$ is equitably $\lceil \Delta(G)/2 \rceil$-list arborable provided $G$ is neither a cycle nor a complete graph of odd order.
\end{conj}

It is shown in~\cite{DF20} that Conjecture~\ref{conj: stronger} holds for 3-dimensional grids, and it is shown in~\cite{LZ21} that Conjecture~\ref{conj: stronger} holds for graphs with treewidth 2 and maximum degree at least 9.  In this paper we will show that Conjecture~\ref{conj: stronger} holds for several additional classes of graphs.  Note that a recent paper~\cite{ZZ21} considers a generalization of equitable list arboricity which generalizes and cites some of our work here (which had already appeared on arXiv.org).

\subsection{Outline of the Paper}
In Section~\ref{lack} we begin by showing that, similar to the relationship between equitable choosability and equitable colorability, it is not the case that for each $k \in \N$ if graph $G$ is equitably $k$-list arborable then $G$ is equitably vertex $k$-arborable.\footnote{When $k \in \{1,2\}$, it is easy to prove that if $G$ is equitably $k$-list arborable, then $G$ is equitably vertex $k$-arborable.}  Specifically, we show that $K_{4,15}$ is not equitably vertex 3-arborable, but it is equitably 3-list arborable.

In~\cite{DD18}, the authors state that they are not aware of any results in the literature that address the following question:  If graph $G$ is equitably $k$-list arborable, must it follow that $G$ is equitably $(k+1)$-list arborable?  We end Section~\ref{lack} by demonstrating that the answer to this question is no.  In particular, we use a combination of probabilistic and algorithmic arguments to show that $K_{11,17}$ is equitably 3-list arborable, but it is not equitably 4-list arborable.

In Section~\ref{conjecture} we verify Conjectures~\ref{conj: HSanalogue} and~\ref{conj: stronger} for complete graphs minus an edge, $2\ell$-regular graphs of order $2\ell+2$, and 2-degenerate graphs $G$ with $\Delta(G) \geq 3$.
The next result will imply that Conjectures~\ref{conj: HSanalogue} and~\ref{conj: stronger} hold for all powers of paths.

\begin{thm} \label{pro: pathpower}
For $n, p \in \N$, suppose that $G= P_{n}^p$.  Then, $G$ is equitably $k$-list arborable for each $k$ satisfying $k \geq p$.
\end{thm}

If $G=P_n^p$ and $n \geq 2p$, then $\lceil \Delta(G)/2 \rceil = p$, so Conjectures~\ref{conj: HSanalogue} and~\ref{conj: stronger} holds for such $G$ by Theorem~\ref{pro: pathpower}.  If $G=P_n^p$ and
$p+1 < n < 2p$, then $\Delta(G)=n-1$ and
$G=P_n^p$ is a spanning subgraph of a complete graph on $n$ vertices minus an edge, so $G$ is equitably $k$-list arborable whenever $k \geq \lceil (n-1)/2 \rceil =\lceil \Delta(G)/2 \rceil$.  If $n\le p+1$ then $P_n^p$ is a complete graph, a known case.

Section~\ref{conjecture} ends by studying Conjecture~\ref{conj: stronger} for graphs of maximum degree at most~4.


In Section~\ref{general} we prove the following result, a generalization of a lemma proven by Zhang in 2016~\cite{Z16} (see Lemma~\ref{lem: pelsmajerlike} in Section~\ref{conjecture} below). This tool helps us to recognize a set $S$ of vertices in a graph $G$ for which an equitable, arborable list coloring of $G-S$ can be extended to an equitable, arborable list coloring of $G$.



\begin{lem} \label{lem: general}
	Suppose $m \in \N$ and $S= \{x_1, \ldots, x_{mk} \}$ where $x_1, \ldots, x_{mk}$ are distinct vertices of $G$.  Suppose that $L$ is a $k$-assignment for $G$, and $L'$ is the $k$-assignment for $G-S$ obtained by restricting the domain of $L$ to $V(G-S)$. Let $f$ be an equitable, arborable $L'$-coloring of $G-S$.
	
Let $D$ be the list assignment of $G[S]$ defined by $D(v) = L(v) - \{ c \in L(v) : |f^{-1}(c) \cap N_G(v)| \geq 2 \}$.	Suppose that there is an arborable $D$-coloring $g$ of $G[S]$ such that (i) $g$ uses no color more than $m$ times and (ii) for each $c \in g(S)$ there is at most one vertex $v \in g^{-1}(c)$ with the property that $c$ is dangerous with respect to $v$.  Then the function $h : V(G) \rightarrow \bigcup_{v \in V(G)} L(v)$ given by
	\[ h(v) = \begin{cases}
	f(v) & \text{if $v \notin S$} \\
	g(v) & \text{if $v \in S$}
	\end{cases}
	\]
	is an equitable, arborable $L$-coloring of $G$.
\end{lem}

Above, color $c \in L(v)$ is \emph{dangerous with respect to $v$} if  $|f^{-1}(c) \cap N_G(v)| = 1$; that is, a neighbor of $v$ is already colored with $c$ by $f$.

We use Lemma~\ref{lem: general} to first improve Theorem~\ref{pro: pathpower} when $p \geq 3$.

\begin{pro} \label{pro: p-1}
Suppose $n,p \in \N$, $p \geq 3$, and $G=P_n^p$.  Then, $G$ is equitably $(p-1)$-list arborable.
\end{pro}

On the other hand, $G=P_n^p$ contains a complete graph on $p+1$ vertices whenever $n\geq p+1$, in which case $G$ is not $k$-arborable for $2k<p+1$ (since it is not even $k$-list arborable), i.e., for $k  < \lceil (p+1)/2 \rceil$.
This observation along with Theorem~\ref{pro: pathpower} and Proposition~\ref{pro: p-1} lead us to make the following conjecture.

\begin{conj} \label{conj: pathpower}
For any $n, p \in \N$ with $n \geq p+1$, $G = P_n^p$ is equitably $k$-list arborable if and only if $k \geq \lceil (p+1)/2 \rceil$.
\end{conj}

Notice that Theorem~\ref{pro: pathpower} and Proposition~\ref{pro: p-1} tell us that Conjecture~\ref{conj: pathpower} is true for $p \in [4]$.  We end Section~\ref{general} by using Lemma~\ref{lem: general} to prove Conjecture~\ref{conj: HSanalogue} for powers of cycles.

\begin{thm} \label{pro: cyclepower}
Suppose that $p \geq 2$ and $n \geq 2p+2$.  If $G = C_{n}^p$, then $G$ is equitably $k$-list arborable for each $k$ satisfying $k \geq p+1$.
\end{thm}

We were not able to improve the bound of $p+1$ to $p$ in Theorem~\ref{pro: cyclepower} which means that Conjecture~\ref{conj: stronger} is still open for powers of cycles.

\section{Some Interesting Examples 
} \label{lack}

We will use the following trivial proposition throughout this section.

\begin{pro} \label{pro: obvious}
Suppose $G$ is a complete bipartite graph with bipartition $X, Y$, and $L$ is a list assignment for $G$.  A mapping $f$ is an arborable $L$-coloring of $G$ if and only if $f$ is an $L$-coloring of $G$ such that for each $c$ in the range of $f$, $|f^{-1}(c) \cap X| \leq 1$ or $|f^{-1}(c) \cap Y| \leq 1$.
\end{pro}

\begin{proof}
The result immediately follows from the fact that any cycle contained in $G$ must have at least two vertices in each of the sets $X$, $Y$, and the fact that any subgraph of $G$ induced by a subset of $V(G)$ with at least two vertices in each of the sets $X$, $Y$ contains a cycle.
\end{proof}

It is well-known that equitable $k$-choosability does not imply equitable $k$-colorability (see e.g.,~\cite{M18} or~\cite{MM19}).  We now show that something similar is true in the context of vertex arboricity.

\begin{pro} \label{pro: noimply}
Let $G=K_{4,15}$.  Then, $G$ is not equitably vertex 3-arborable, and $G$ is equitably 3-list arborable.
\end{pro}

\begin{proof}
Throughout the proof suppose that the bipartition of $G$ is $A = \{a_1, a_2, a_3, a_4 \}$, $B = \{b_1, b_2, \ldots, b_{15} \}$.

First, we prove that $G$ is not equitably vertex 3-arborable.  For the sake of contradiction, suppose that $f: V(G) \rightarrow [3]$ is an arborable 3-coloring of $G$ which means $|f^{-1}(i)|$ is 6 or 7 for each $i \in [3]$.  Without loss of generality, suppose that $|f^{-1}(1) \cap A| \geq |f^{-1}(2) \cap A| \geq |f^{-1}(3) \cap A|$.  The Pigeonhole Principle and the fact that $|A| = 4$ implies $2 \leq |f^{-1}(1) \cap A| \leq 4$.  Finally, the fact that $|f^{-1}(1)| \geq 6$ implies that $2 \leq |f^{-1}(1) \cap B|$.  Thus, $G[f^{-1}(1)]$ contains a cycle which is a contradiction.

Now, we will prove that $G$ is equitably $3$-list arborable.  We note that $\lceil |V(G)|/3 \rceil = 7$, and we suppose that $L$ is an arbitrary 3-assignment for $G$.  We will now construct an equitable, arborable $L$-coloring of $G$.  Notice that if there is no color in $\bigcup_{v \in V(G)} L(v)$ that appears in at least 7 of the lists: $L(b_1), \ldots, L(b_{15})$, we can begin by greedily coloring the vertices $a_1, \ldots, a_4$ with colors from their respective lists so that no color is used more than twice and at most one color is used twice.  Then, we can greedily color the vertices $b_1, \ldots, b_{15}$ with colors from their respective lists so that the color used twice on the vertices in $A$ (if there is such a color) is not used to color any vertices in $B$.  The resulting coloring is clearly an arborable $L$-coloring of $G$ that uses no color more than 7 times.

So, we may suppose without loss of generality that there is a color $c$ such that $c \in L(b_i)$ for each $i \in [7]$.  Suppose we color each vertex in $\{b_i : i \in [7] \}$ with $c$.  Then, for each $v \in V(G) - \{b_i : i \in [7] \}$, let $L'(v) = L(v) - \{c \}$.  Notice that if there is a color $d$ in at least 7 of the lists: $L'(b_8), \ldots, L'(b_{15})$ we can complete an equitable, arborable $L$-coloring of $G$ as follows. Begin by coloring 7 of the vertices in $\{b_8, \ldots, b_{15} \}$ that have $d$ in their list with $d$, and then greedily color the 5 remaining uncolored vertices with a color in their respective lists so that none of these remaining 5 vertices get colored with $c$ or $d$.

So, we may suppose that no color in $\bigcup_{v \in V(G) - \{b_i : i \in [7] \}} L'(v)$ appears in at least 7 of the lists: $L'(b_8), \ldots, L'(b_{15})$.  Notice that if there is a $d \in L'(a_i)$ for each $i \in [4]$, we can color each vertex in $A$ with $d$.  Then, we can greedily color each of the vertices $b_8, \ldots, b_{15}$ with a color in their respective lists so that none of these remaining vertices get colored with $c$ or $d$. The resulting coloring is clearly an arborable (in fact proper) $L$-coloring of $G$ that uses no color more than 7 times.  So, we may suppose that $\bigcap_{i=1}^4 L'(a_i) = \emptyset$.  Since $\bigcap_{i=1}^4 L'(a_i) = \emptyset$, it is possible to color each $a_i \in A$ with a color from $L'(a_i)$ so that no color is used more than twice and at most one color is used twice in coloring the vertices in $A$. Then, we can greedily color the vertices $b_8, \ldots, b_{15}$ with colors from their respective lists so that $c$ and the color used twice on the vertices in $A$ (if there is such a color) is not used to color any vertex of these vertices. The resulting coloring is clearly an arborable $L$-coloring of $G$ that uses no color more than 7 times.
\end{proof}

We will now work towards showing that if graph $G$ is equitably $k$-list arborable, $G$ need not be equitably $(k+1)$-list arborable.  In particular, we will show that $K_{11,17}$ is equitably 3-list arborable, but it is not equitably 4-list arborable.  To do this we begin by proving two general propositions, and a lemma.

\begin{pro} \label{pro: prob}
Let $G = K_{n,m}$ and $k \in \N$.  If $n+m \leq (k+1)2^k - 1$, then $G$ is $k$-list arborable.
\end{pro}

\begin{proof}
Suppose that $L$ is an arbitrary $k$-assignment for $G$, and suppose $G$ has bipartition $X$, $Y$ with $|X|=n$ and $|Y|=m$.  Suppose we construct the sets $C_X$ and $C_Y$ via the following random process.  For each $c \in \bigcup_{v \in V(G)} L(v)$ flip a fair coin.  If the coin lands heads place $c$ in $C_X$; otherwise, place $c$ in $C_Y$.  After this process has concluded, for each $v \in X$ if $L(v) \cap C_X \neq \emptyset$, color $v$ with an element in $L(v) \cap C_X$.  Similarly, for each $v \in Y$ if $L(v) \cap C_Y \neq \emptyset$, color $v$ with an element in $L(v) \cap C_Y$.  Clearly, the resulting (perhaps partial) $L$-coloring is proper.

Now, for each $v \in V(G)$, let $X_v$ be the random variable that is equal to 1 if $v$ is uncolored and equal to 0 if $v$ is colored.  So, $\sum_{v \in V(G)} X_v$ is the number of vertices in $G$ that are uncolored.  Clearly, $P[X_v = 1] = (1/2)^k$.  So, by linearity of expectation,
$$\E \left[ \sum_{v \in V(G)} X_v \right] = \frac{n+m}{2^k}.$$
Since $n+m \leq (k+1)2^k - 1$, we see that $\lfloor \E[ \sum_{v \in V(G)} X_v] \rfloor \leq k$.

Since $\sum_{v \in V(G)} X_v$ is always an integer, there is a partial, proper $L$-coloring of $G$, $f$, that leaves at most $k$ vertices uncolored.  We can extend $f$ to an arborable $L$-coloring of $G$ by coloring the vertices outside of the domain of $f$ with pairwise distinct colors from their respective lists (this is possible since there are at most $k$ uncolored vertices and each list contains $k$ colors).  Then, our resulting $L$-coloring is an arborable $L$-coloring of $G$ by Proposition~\ref{pro: obvious}.
\end{proof}

\begin{pro} \label{pro: twocolors}
Let $G$ be a complete bipartite graph with bipartition $X$, $Y$, and suppose $L$ is a 2-assignment for $G$ such that there is no arborable $L$-coloring of $G$.  Suppose there exists a partial $L$-coloring $f: X \rightarrow \bigcup_{v \in X} L(v)$ satisfying the following two conditions: (1) there are two distinct colors $a$ and $b$ satisfying $|f^{-1}(a)| \geq 2$ and $|f^{-1}(b)| \geq 2$ and (2) for each $c \in  \bigcup_{v \in X} L(v) - \{a,b\}$, $|f^{-1}(c)| \leq 1$.  Then, $|L^{-1}(\{a,b\}) \cap Y| \geq 3$.
\end{pro}

\begin{proof}
For the sake of contradiction suppose that $|L^{-1}(\{a,b\}) \cap Y| \leq 2$.  Now, we obtain a contradiction by extending $f$ to an arborable $L$-coloring of $G$ as follows.  For each $v \in Y$ such that $v \notin L^{-1}(\{a,b\})$, color $v$ with a color in $L(v) - \{a,b\}$.  Then, color the vertices in $L^{-1}(\{a,b\}) \cap Y$ with distinct colors (if there are any such vertices).  Our resulting $L$-coloring is an arborable $L$-coloring of $G$ by Proposition~\ref{pro: obvious}.
\end{proof}

\begin{lem} \label{lem: K711}
Let $G = K_{7,11}$ and $L$ be a 2-assignment for $G$.  Then, there is an arborable $L$-coloring of $G$ that uses no color more than 10 times.
\end{lem}

\begin{proof}
Suppose the bipartition of $G$ is $X= \{x_1, \ldots, x_7\}$, $Y = \{y_1, \ldots, y_{11} \}$.  For the sake of contradiction, suppose that $L$ is a 2-assignment for $G$ such that there is no arborable $L$-coloring of $G$ that uses no color more than 10 times.  For each $A \in \{X, Y\}$, let $\eta_A : \bigcup_{v \in A}L(v) \rightarrow \N$ be the function given by $\eta_A(c) = |\{v \in A: c \in L(v) \}|$.  Additionally, for each $A \in \{X, Y\}$, let
$$m_A = \max_{c \in \bigcup_{v \in A} L(v)} \eta_A(c).$$
Note that if $m_Y \geq 10$, we can assume without loss of generality that there is a color $c \in L(y_i)$ for each $i \in [10]$.  We can complete an arobrable $L$-coloring of $G$ that uses no color more than 10 times by coloring $y_1, \ldots, y_{10}$ with $c$ and then coloring each $v \in V(G) - \{y_i: i\in [10]\}$ with a color in $L(v) - \{c\}$.  So, we know that $m_Y \leq 9$, and the following observation is clear.
\\
\\
\emph{Observation:}  If it is possible to color the vertices in $X$ with colors from their respective lists in such a way that at most one color, $c$, is used more than once, then by Proposition~\ref{pro: obvious} we can obtain an arborable $L$-coloring of $G$ by coloring each $v \in Y$ with a color in $L(v) - \{c\}$.  This coloring cannot use a color more than 10 times since $m_Y \leq 9$.  So, it is impossible to color the vertices in $X$ with colors from their respective lists in such a way that at most one color is used more than once.  So, at least two colors must be used at least two times in coloring the vertices in $X$.
\\
\\
\indent Now, we know that $m_X$ must equal some element in $[7]$.  We will now obtain a contradiction in each of these 7 cases.  We now pursue each of the seven cases in increasing order of difficulty.

In the case where $m_X \geq 6$, we can assume without loss of generality that there is a color $c \in L(x_i)$ for each $i \in [6]$.  We can color each of $x_1, \ldots, x_{6}$ with $c$ and then color $x_7$ with a color in $L(x_7) - \{c\}$.  This contradicts our Observation.

In the case where $m_X = 5$, we can assume without loss of generality that there is a color $c \in L(x_i)$ for each $i \in [5]$.  We also know that $c \notin L(x_i)$ when $i = 6, 7$.  So, we can color $x_1, \ldots, x_5$ with $c$.  Then, we can color $x_6, x_7$ with distinct colors from $L(x_6)$ and $L(x_7)$.  This contradicts our Observation.

In the case where $m_X \leq 2$, suppose we independently and randomly color each vertex $v \in X$ with a color from $L(v)$ such that each color from $L(v)$ has an equal chance of being chosen.  For each $c \in \bigcup_{v \in X} L(v)$, let $X_c$ be the random variable that is equal to 1 if the color $c$ is used twice in coloring the vertices of $X$ and equal to 0 otherwise.  Notice that when $\eta_X(c) = 2$, we have that $P[X_c = 1] = 1/4$, and when $\eta_X(c) = 1$, we have that $P[X_c = 1]=0$.  Since $\sum_{c \in \bigcup_{v \in X} L(v)} \eta_X(c) = 14$, there are at most seven elements in $\bigcup_{v \in X} L(v)$ that appear in two of the lists: $L(x_1), \ldots, L(x_7)$.  Consequently, $\E \left [ \sum_{c \in \bigcup_{v \in X} L(v)} X_c \right ] \leq \frac{7}{4}.$  Thus, there is a way to color the vertices in $X$ with colors from their respective lists such that at most one color is used more than once.  This contradicts our Observation.

Now, suppose that $m_X=4$.  We can assume without loss of generality that there is a color $c \in L(x_i)$ for each $i \in [4]$.  Now, we claim it must be the case that $L(x_5)=L(x_6)=L(x_7)$.  To see why, notice that if this was not the case, we could color each of the vertices: $x_1, \ldots, x_4$ with $c$, and we could color $x_5, x_6, x_7$ with pairwise distinct colors from $L(x_5)$, $L(x_6)$, $L(x_7)$ respectively which contradicts our Observation.  So, we may assume that $L(x_5)=L(x_6)=L(x_7)= \{c_1, c_2 \}$ where $c, c_1, c_2$ are pairwise distinct.  We will now consider two sub-cases: (1) $|\bigcup_{i=1}^4 L(x_i)| \geq 4$ and (2) $|\bigcup_{i=1}^4 L(x_i)| \leq 3$.

In sub-case (1) we can color $x_1, \ldots, x_4$ with pairwise distinct colors from $L(x_1)$, $L(x_2)$, $L(x_3)$, $L(x_4)$ respectively.  Then, we can color each of $x_5, x_6, x_7$ with $c_1$ which contradicts our Observation.

In sub-case (2) we can assume without loss of generality that $L(x_i) = \{c,c_3\}$ for each $i \in [2]$.  Clearly, $c, c_1, c_2, c_3$ are pairwise different (since each of them appears in at most 4 of the lists: $L(x_1), \ldots, L(x_7)$).  Also, we know that if $\eta_X(c_3) \geq 3$, then we may assume $L(x_3) = \{c,c_3 \}$; otherwise, we know $L(x_3) \neq L(x_2)$ and $L(x_3)=L(x_4)$.  It is not hard to see that it is possible to $L$-color the vertices in $X$ in the following four different ways: (1) $c$ is used 4 times and $c_1$ is used 3 times, (2) $c$ is used 4 times and $c_2$ is used 3 times, (3) $c_3$ is used 2 or 3 times, $c_1$ is used 3 times, and no other color is used more than once, and (4) $c_3$ is used 2 or 3 times, $c_2$ is used 3 times, and no other color is used more than once.  We know that none of these four partial $L$-colorings is extendable to an arborable $L$-coloring of $G$.  By Proposition~\ref{pro: twocolors} we have $|L^{-1}(\{c,c_1\}) \cap Y| \geq 3$, $|L^{-1}(\{c,c_2\}) \cap Y| \geq 3$, $|L^{-1}(\{c_3,c_1\}) \cap Y| \geq 3$, and $|L^{-1}(\{c_3,c_2\}) \cap Y| \geq 3$.  This implies $|Y| \geq 12$ which is a contradiction.

Finally, we turn our attention to the case where $m_X = 3$.  Let $a$ be the number of elements $c \in \bigcup_{v \in X} L(v)$ that satisfy $\eta_X(c) = 3$.  First, we claim that $a \geq 2$.  To see why this is so, suppose $a=1$.  Then, independently and randomly color each vertex $v \in A$ with a color from $L(v)$ such that each color from $L(v)$ has an equal chance of being chosen.  For each $c \in \bigcup_{v \in X} L(v)$, let $X_c$ be the random variable that is equal to 1 if the color $c$ is used at least twice in coloring the vertices of $X$ and equal to 0 otherwise.  Notice that: when $\eta_X(c)=3$, we have that $P[X_c = 1] = 1/2$, when $\eta_X(c) = 2$, we have that $P[X_c = 1] = 1/4$, and when $\eta_X(c) = 1$, we have that $P[X_c = 1]=0$.  Since $\sum_{c \in \bigcup_{v \in X} L(v)} \eta_X(c) = 14$ and $a=1$, there are at most five elements in $\bigcup_{v \in X} L(v)$ that appear in two of the lists: $L(x_1), \ldots, L(x_7)$.  Consequently, $\E \left [ \sum_{c \in \bigcup_{v \in X} L(v)} X_c\right ] \leq \frac{1}{2}+\frac{5}{4} = \frac{7}{4}.$ So, we are able to proceed as we did in the case where $m_X \leq 2$.  This means we may assume that $a \geq 2$.  Now, we claim that there is no $c \in \bigcup_{v \in X} L(v)$ satisfying $\eta_X(c)=1$.  To see why this is so, suppose color $o$ is such a color.  Suppose without loss of generality that $L(x_1) = \{o,c\}$.  Now, let $L'$ be the 2-assignment for $G$ obtained from $L$ by replacing the $o$ in $L(x_1)$ with a color $d \neq c$ satisfying $\eta_X(d)=3$.  From the argument used for the case where $m_X=4$, we know that we can obtain an arborable $L'$-coloring $f$ of $G$ that uses no color more than 10 times.  Now, if $f(x_1)=d$ modify $f$ by recoloring $x_1$ with $o$; otherwise, do not modify $f$.  The resulting coloring is clearly an arborable $L$-coloring of $G$ that uses no color more than 10 times which is a contradiction.

Now, let $b$ be the number of elements $c \in \bigcup_{v \in X} L(v)$ that satisfy $\eta_X(c) = 2$.  By what we have shown thus far, we know that $3a+2b = 14$ and $a \geq 2$.  So, we will derive a contradiction in each of the following sub-cases to complete the proof: (1) $a=2$ and $b=4$ and (2) $a=4$ and $b=1$.  In each of these sub-cases we will assume without loss of generality that the color $c_1$ is such that $c_1 \in L(x_i)$ for each $i \in [3]$.  Also, we let $X' = \{x_4, x_5, x_6, x_7 \}$, and we let $\eta_{X'} : \bigcup_{v \in X'} L(v) \rightarrow \N$ be the function given by $\eta_{X'}(c) = |\{v \in X': c \in L(v) \}|$.

In sub-case (1) begin by coloring $x_1, x_2$, and $x_3$ with $c_1$.  If there is a color $o$ such that $\eta_{X'}(o)=1$, we assume without loss of generality that $o \in L(x_4)$ and we color $x_4$ with $o$.  Since $a=2$ it is not possible that $L(x_5)=L(x_6)=L(x_7)$.  So, we can color $x_5, x_6, x_7$ with pairwise distinct colors from $L(x_5)$, $L(x_6)$, and $L(x_7)$.  This contradicts our Observation.  So, we may assume that $\min_{c \in \bigcup_{v \in X'} L(v)} \eta_{X'}(c) \geq 2$.  Since $\max_{c \in \bigcup_{v \in X'} L(v)} \eta_{X'}(c) \leq 3$, $a=2$, and $\sum_{c \in \bigcup_{v \in X'} L(v)} \eta_{X'} (c) = 8$, it must be that the domain of $\eta_{X'}$ is of size 4 and $\eta_{X'}$ outputs 2 for each element in its domain.  It is then easy to see that we can color $x_4, x_5, x_6, x_7$ with pairwise distinct colors from $L(x_4)$, $L(x_5)$, $L(x_6)$, and $L(x_7)$ respectively (simply consider the case where the lists $L(x_4), L(x_5), L(x_6), L(x_7)$ are pairwise distinct and the case where they are not pairwise distinct).  This however contradicts our Observation.

In sub-case (2) we may suppose that $c_1, c_2, c_3, c_4, d$ are pairwise distinct colors such that $\eta_X(c_i)=3$ for each $i \in [4]$ and $\eta_X(d)=2$.  Since $\sum_{c \in \bigcup_{v \in X'} L(v)} \eta_{X'} (c) = 8$, we can complete sub-case (2) by considering the three following situations: (a) there is a color $z$ such that $\eta_{X'}(z)=1$ (note: $z \in \{c_1, c_2, c_3, c_4, d \}$), (b) the domain of $\eta_{X'}$ is of size 4 and $\eta_{X'}$ outputs 2 for each element in its domain, and (c) the domain of $\eta_{X'}$ is of size 3, $|\eta_{X'}^{-1}(3)|=2$, and $|\eta_{X'}^{-1}(2)|=1$.

For (a) we assume without loss of generality that $z \in L(x_4)$.  Now, if it is not the case that $L(x_5)=L(x_6)=L(x_7)$ we can proceed as we did at the beginning of sub-case (1).  So, we may assume that $L(x_5)=L(x_6)=L(x_7) = \{c_2, c_3\}$.  This means that we can assume without loss of generality that $L(x_4)=\{c_4,d \}$, $L(x_3)=L(x_2)=\{c_1,c_4\}$, and $L(x_1)= \{c_1,d\}$.  It is now clear that it is possible to $L$-color the vertices in $X$ in the following four different ways: (1) $c_1$ is used 3 times, $c_2$ is used 3 times, and $d$ is used once, (2) $c_1$ is used 3 times, $c_3$ is used 3 times, and $d$ is used once, (3) $c_4$ is used 3 times, $c_2$ is used 3 times, and $d$ is used once, and (4) $c_4$ is used 3 times, $c_3$ is used 3 times, and $d$ is used once.  We know that none of these four partial $L$-colorings is extendable to an arborable $L$-coloring of $G$.  By Proposition~\ref{pro: twocolors} we have $|L^{-1}(\{c_1,c_2\}) \cap Y| \geq 3$, $|L^{-1}(\{c_1,c_3\}) \cap Y| \geq 3$, $|L^{-1}(\{c_2,c_4\}) \cap Y| \geq 3$, and $|L^{-1}(\{c_3,c_4\}) \cap Y| \geq 3$.  This implies $|Y| \geq 12$ which is a contradiction.

For (b) we may proceed as we did at the end of sub-case (1).

For (c) we may assume without loss of generality that $\eta_{X'}(c_2) = \eta_{X'}(c_3) = 3$.  Then, it must be that $\eta_{X'}(d)=2$ or $\eta_{X'}(c_4) = 2$.  If $\eta_{X'}(d) = 2$, we can assume without loss of generality that $L(x_1)=L(x_2)=L(x_3) = \{c_1, c_4 \}$, $L(x_4)=L(x_5)= \{c_2,c_3\}$, $L(x_6) = \{c_2,d\}$, and $L(x_7) = \{c_3,d\}$.  Then, we can get a contradiction by proceeding in a fashion like situation (a).  Finally, if $\eta_{X'}(c_4)=2$, we can assume without loss of generality that $L(x_1)=L(x_2) = \{c_1, d \}$, $L(x_3)=\{c_1,c_4\}$, $L(x_4)=L(x_5)= \{c_2,c_3\}$, $L(x_6) = \{c_2,c_4\}$, and $L(x_7) = \{c_3,c_4\}$.  It is now clear that it is possible to $L$-color the vertices in $X$ in the following four different ways: (1) $c_1$ is used 3 times, $c_2$ is used 3 times, and $c_4$ is used once, (2) $c_1$ is used 3 times, $c_3$ is used 3 times, and $c_4$ is used once, (3) $d$ is used 2 times, $c_2$ is used 3 times, $c_1$ is used once, and $c_4$ is used once, and (4) $d$ is used 2 times, $c_3$ is used 3 times, $c_1$ is used once, and $c_4$ is used once.  We know that none of these four partial $L$-colorings is extendable to an arborable $L$-coloring of $G$.  By Proposition~\ref{pro: twocolors} we have $|L^{-1}(\{c_1,c_2\}) \cap Y| \geq 3$, $|L^{-1}(\{c_1,c_3\}) \cap Y| \geq 3$, $|L^{-1}(\{d,c_2\}) \cap Y| \geq 3$, and $|L^{-1}(\{d,c_3\}) \cap Y| \geq 3$.  This implies $|Y| \geq 12$ which is a contradiction.
\end{proof}

We are finally ready to show that $K_{11,17}$ is equitably 3-list arborable, but it is not equitably 4-list arborable.

\begin{thm} \label{thm: notmonotone}
Let $G=K_{11,17}$.  Then, $G$ is equitably 3-list arborable, but $G$ is not equitably 4-list arborable.
\end{thm}

\begin{proof}
Throughout the proof suppose the bipartition of $G$ is $X= \{x_1, \ldots, x_{11}\}$, $Y = \{y_1, \ldots, y_{17} \}$.  First, we will show that $G$ is not equitably 4-list arborable by constructing a 4-assignment $L$ for $G$ for which there is no equitable, arborable $L$-coloring of $G$.  Suppose $L$ is the 4-assignment for $G$ that assigns the list $\{1,2,3,4\}$ to every vertex.  For the sake of contradiction, suppose that $f$ is an equitable, arborable $L$-coloring of $G$.  We know that $|f^{-1}(i)| \leq 7$ for each $i \in [4]$.  This along with the fact that $|V(G)|=28$ implies $|f^{-1}(i)| = 7$ for each $i \in [4]$.  Without loss of generality, suppose that $|f^{-1}(1) \cap X| \geq |f^{-1}(2) \cap X| \geq |f^{-1}(3) \cap X| \geq |f^{-1}(4) \cap X|$.  By the fact that $|X|=11$ and the Pigeonhole Principle, $2 \leq |f^{-1}(2) \cap X| \leq 5$.  So, $|f^{-1}(2) \cap Y| \geq 2$, and consequently $f$ is not an arborable $L$-coloring of $G$ by Proposition~\ref{pro: obvious} which is a contradiction.

Now, we will show that $G$ is equitably 3-list arborable.  Suppose that $L$ is an arbitrary $3$-assignment for $G$.  We will show that an arborable $L$-coloring of $G$ that uses no color more than $\lceil |V(G)|/3 \rceil$ times exists.  Since $|V(G)|=28 \leq 2^3(4)-1$, we know that an arborable $L$-coloring of $G$ exists by Proposition~\ref{pro: prob}.  So, if each $c \in \bigcup_{v \in V(G)} L(v)$  has the property that it appears in no more than 9 of the lists: $L(x_1), \ldots, L(x_{11})$ and no more than 9 of the lists: $L(y_1), \ldots, L(y_{17})$, we are done since any arborable $L$-coloring of $G$ will also be equitable if this holds.

So, we just need to construct an equitable, arborable $L$-coloring of $G$ in each of the following cases: (1) there is a $c \in \bigcup_{v \in V(G)} L(v)$ such that $c$ appears in at least 10 of the lists: $L(x_1), \ldots, L(x_{11})$ or (2) there is a $c \in \bigcup_{v \in V(G)} L(v)$ such that $c$ appears in at least 10 of the lists: $L(y_1), \ldots, L(y_{17})$.  In case (1) we can assume without loss of generality that $c \in L(x_i)$ for each $i \in [10]$.  Suppose we color each of $x_1, \ldots, x_{10}$ with $c$.  Then, for each $v \in V(G) - \{x_i: i \in [10]\}$, let $L'(v) = L(v) - \{c\}$.  Now, we can greedily color each vertex $x_{11}, y_1, \ldots, y_{17}$ with a color assigned to the vertex by $L'$ in such a way that no color is used more than 10 times (the coloring need not be proper).  This completes an equitable, arborable $L$-coloring of $G$.

For case (2) assume without loss of generality that $c \in L(y_i)$ for each $i \in [10]$.  Suppose we color each of $y_1, \ldots, y_{10}$ with $c$.  Then, for each $v \in V(G) - \{y_i: i \in [10]\}$, let $L'(v) = L(v) - \{c\}$, and note that $|L'(v)| \geq 2$.  Lemma~\ref{lem: K711} then implies that there is an arborable $L'$-coloring of $G - \{y_i: i \in [10]\}$ that uses no color more than 10 times.  Such a coloring completes an equitable, arborable $L$-coloring of $G$.
\end{proof}

\section{Verifying Conjecture~\ref{conj: stronger} for Certain Graphs} \label{conjecture}

We begin by verifying Conjecture~\ref{conj: stronger} for graphs with high maximum degree. 

\begin{pro} \label{pro: almostcomplete}
If $n \geq 3$ and $G = K_n - e$ where $e \in E(G)$, then $G$ is equitably $k$-list arborable whenever $k \geq \lceil \Delta(G)/2 \rceil$.
\end{pro}

\begin{proof}
The result clearly follows when $n$ is even since we know: Conjecture~\ref{conj: HSanalogue} holds for complete graphs, $G$ is a spanning subgraph of a complete graph on $n$ vertices, and when $n$ is even, $\lceil \Delta(G)/2 \rceil = \lceil (n-1)/2 \rceil = n/2 = \lceil (\Delta(K_n) + 1)/2 \rceil$.  So, we may suppose that $n = 2\ell+1$ where $\ell \in \N$.  Similar to when $n$ is even, since $G$ is a spanning subgraph of a complete graph on $n$ vertices, we know that $G$ is equitably $k$-list arborable whenever $k \geq \lceil (\Delta(K_n) + 1)/2 \rceil = \ell+1$.  So, to complete the proof we need only show that $G$ is equitably $\ell$-list arborable.  Suppose that $L$ is an arbitrary $\ell$-assignment for $G$.  We will now construct an equitable, arborable $L$-coloring $f$ of $G$ to complete the proof.

Suppose $V(G) = \{v_1, v_2, \ldots, v_{2\ell+1} \}$, and without loss of generality, assume that $v_1v_{2\ell+1} \notin E(G)$.  Begin by greedily coloring $v_1, \ldots, v_{2\ell}$ so that for each $i \in [2\ell]$, $f(v_i) \in L(v_i)$ and no color is used more than twice in coloring $v_1, \ldots, v_{2\ell}$.  Then, let
$$L'(v_{2\ell+1}) = L(v_{2\ell+1}) - \{c \in L(v_{2\ell+1}) : \text{$c$ is used twice in coloring $v_1, \ldots, v_{2\ell}$} \}.$$
Since $d_{G}(v_{2\ell+1}) = 2\ell-1$ and $|L(v_{2\ell+1})|=\ell$, we know that $|L'(v_{2\ell+1})| \geq 1$.  So, we can color $v_{2\ell+1}$ so that $f(v_{2\ell+1}) \in L'(v_{2\ell+1})$.  By construction, $f$ uses no color more than $\lceil |V(G)|/\ell \rceil = 3 $ times, and the only color that can be used 3 times by $f$ is $f(v_{2\ell+1})$.  However, since $v_{2\ell+1}$ can only be adjacent to at most one vertex colored with $f(v_{2\ell+1})$, we know $G[f^{-1}(f(v_{2\ell+1}))]$ is acyclic.  Consequently, $f$ is an equitable, arborable $L$-coloring of $G$, and $G$ is equitably $\ell$-list arborable.
\end{proof}

\begin{pro} \label{pro: small}
Suppose $G$ is a $2\ell$-regular graph with $\ell \geq 2$.  If $|V(G)|=2\ell+2$, then $G$ is equitably $k$-list arborable whenever $k \geq \ell$.
\end{pro}

\begin{proof}
The result is obvious when $k > \ell$.  So, we need only show that $G$ is equitably $\ell$-list arborable.

For the sake of contradiction, suppose that $G$ is not equitably $\ell$-list arborable.  Let $L$ be an $\ell$-assignment for $G$ for which there is no equitable, arborable $L$-coloring of $G$.  Since $G$ is not a complete graph, by Theorem~\ref{thm: upperarbor}, we know that there is an arborable $L$-coloring of $G$.  Suppose $f$ is an arborable $L$-coloring of $G$.  Since $f$ is not an equitable, arborable $L$-coloring of $G$, we know that there is a $c \in \bigcup_{v \in  V(G)} L(v)$ such that $|f^{-1}(c)| > \lceil |V(G)|/l \rceil = 3$.  Let $G' = G[f^{-1}(c)]$.  Note that for each $v \in V(G')$, $v$ is not adjacent to exactly one vertex in $G$.  So, we have that $d_{G'}(v) \geq |V(G')|-2 \geq 2$ which implies that $G'$ contains a cycle.  This however contradicts the fact that $f$ is an arborable $L$-coloring of $G$.
\end{proof}

We now improve upon a result in~\cite{Z16} and verify Conjecture~\ref{conj: stronger} for 2-degenerate graphs.  Our proof uses the following Lemma.

\begin{lem} [\cite{Z16}] \label{lem: pelsmajerlike}
Suppose $S= \{x_1, \ldots, x_k \}$ where $x_1, \ldots, x_k$ are distinct vertices of $G$.  Suppose that $L$ is a $k$-assignment for $G$, and $L'$ is the $k$-assignment for $G-S$ obtained by restricting the domain of $L$ to $V(G-S)$.  If an equitable, arborable $L'$-coloring of $G-S$ exists and
$$|N_G(x_i) - S| \leq 2i-1$$
for each $i \in [k]$, then an equitable, arborable $L$-coloring of $G$ exists.
\end{lem}

\begin{thm} \label{thm: 2degen}
If $G$ is a 2-degenerate graph with $\Delta(G) \geq 3$, then $G$ is equitably $k$-list arborable whenever $k \geq \lceil \Delta(G)/2 \rceil$.
\end{thm}

\begin{proof}
Our proof is by induction on $|V(G)|$.  Note that the result is clear when $|V(G)| = 4$.  So, assume that $|V(G)| > 4$ and the desired result holds for all graphs having less than $|V(G)|$ vertices.  Suppose $k$ satisfies $k \geq \lceil \Delta(G)/2 \rceil$ (we may assume that $k < |V(G)|$ since $G$ is clearly equitably $k$-list arborable whenever $k \geq |V(G)|$).

Let $L$ be an arbitrary $k$-assignment for $G$.  Suppose that $uv \in E(G)$ and $d_G(u) \leq 2$.  Let $x_1=u$ and $x_k = v$.  We construct a subset $S$ of $V(G)$ via the following inductive process.  Begin by placing $x_1$ and $x_k$ in $S$.  Then if $k \geq 3$, for each $i \in \{2, \ldots, k-1 \}$ let $x_i$ be a vertex of degree at most 2 in the graph $G - \{x_1, \ldots, x_{i-1},x_k \}$ (such an $x_i$ exists since $G$ is 2-degenerate).  Now, consider $G-S$ and the $k$-assignment $L'$ for $G-S$ obtained by restricting the domain of $L$ to $V(G-S)$.  Note $G-S$ is 2-degenerate and $\Delta(G-S) \leq \Delta(G)$.  If $\Delta(G-S) \leq 2$, then an equitable, arborable $L'$-coloring of $G-S$ exists by Theorem~5 in~\cite{Z16} (since $\lceil \Delta(G)/2 \rceil \geq 2 \geq \lceil (\Delta(G-S)+1)/2 \rceil$).  If $\Delta(G-S) \geq 3$, then an equitable, arborable $L'$-coloring of $G-S$ exists by the inductive hypotheses.

When $k \geq 3$, it is clear that for each $i \in \{2, \ldots, k-1 \}$, $|N_G(x_i) - S| \leq 2 \leq 2i-1$.  Also, for all possible $k$, $|N_G(x_1)-S| \leq 1 = 2(1)-1$ and $|N_G(x_k) - S| \leq \Delta(G)-1 \leq 2k-1$.  So, an equitable, arborable $L$-coloring of $G$ exists by Lemma~\ref{lem: pelsmajerlike}, and we have that $G$ is equitably $k$-list arborable.
\end{proof}

The next theorem along with Proposition~\ref{pro: almostcomplete} shows Conjecture~\ref{conj: stronger} holds for powers of paths.

\begin{customthm} {\bf \ref{pro: pathpower}}
For $n, p \in \N$, let $G= P_{n}^p$. Then, $G$ is equitably $k$-list arborable for each $k$ satisfying $k \geq p$.
\end{customthm}

\begin{proof}
Suppose $p,k$ are fixed natural numbers satisfying $k \geq p$.  We will prove that $G$ is equitably $k$-list arborable by induction on $n$.  For the basis step note that when $n \leq 2k-1$ the result is clear.

So, suppose that $n \geq 2k$, and the desired result holds for all natural numbers less than $n$.  Suppose that the vertices of the underlying copy of $P_n$ used to form $G$ (in order) are: $v_1, \ldots, v_n$.  Suppose also that $L$ is an arbitrary $k$-assignment for $G$, and $S= \{v_1, \ldots, v_k \}$.  Then, the inductive hypothesis tells us that an equitable, arborable $L'$-coloring of $G-S$ exists where $L'$ is the $k$-assignment for $G-S$ obtained by restricting the domain of $L$ to $V(G-S)$.  Note that for each $i \in [k]$,
$$|N_G(v_i) - S| = \max \{0, p - (k-i) \} \leq i \leq 2i-1.$$
Thus, an equitable, arborable $L$-coloring of $G$ exists by Lemma~\ref{lem: pelsmajerlike}.  It immediately follows that $G$ is equitably $k$-list arborable.
\end{proof}

We end this section by focusing on Conjecture~\ref{conj: stronger} for graphs with maximum degree at most~4.

\begin{thm} \label{thm: four}
Suppose $G$ is a graph with $\Delta(G) \leq 4$.  Then, the following statements hold. \\
(i)  If $G$ has at most 3 vertices of degree 4, then $G$ is equitably 2-list arborable.
\\
(ii)  If $G$ is connected, 4-regular, and $|V(G)| \in \{6,7,8,9,10,11,13,15 \}$, then $G$ is equitably 2-list arborable.
\end{thm}

\begin{proof}
Throughout our proof we let $|V(G)|=n$.  For Statement~(i), suppose that $G$ is a counterexample to the desired statement with fewest number of vertices.  Clearly, $n \geq 5$.  Suppose that $L$ is a 2-assignment for $G$ for which there is no equitable, arborable $L$-coloring of $G$.  There must be a vertex $v \in V(G)$ with $d_G(v) \leq 3$.  Let $G' = G- \{v\}$, and let $L'$ be the 2-assignment for $G'$ obtained by restricting the domain of $L$ to $V(G')$.  By the minimality of $G$, we know that there is an equitable, arborable $L'$-coloring of $G'$ which we will call $f'$.

Now, there must be some $a \in L(v)$ such that $|f'^{-1}(a) \cap N_G(v)| \leq 1$.  So, we let $f$ be the $L$-coloring of $G$ given by $f(v)=a$ and $f(u) = f'(u)$ when $u \in V(G')$.  Clearly, $f$ is an arborable $L$-coloring of $G$.  Since $f$ cannot be an equitable, arborable $L$-coloring of $G$ it must be that
$$\left \lceil \frac{n}{2} \right \rceil < |f^{-1}(a)| = |f'^{-1}(a)| + 1.$$
Since $|f'^{-1}(a)| \leq \lceil (n-1)/2 \rceil$, it must be the case that $n=2l$ for some $l \in \N$ and $|f^{-1}(a)|=l+1$.

We let $A = f^{-1}(a)$ and $B = V(G) - A$.  Note that $|B| = l-1$ and no color other than $a$ is used by $f$ more than $l-1$ times.  We also let $G_1$ be the forest $G[A]$ and $G_2 = G[B]$.  Notice that for each $u \in A$, $u$ is adjacent to two vertices $x, y \in B$ (in $G$) such that: $f(x) \in L(u) - \{a\}$, $f(x)=f(y)$, and there is a path in $G[f^{-1}(f(x))]$ connecting $x$ and $y$.  This is because if this was not the case, we could recolor $u$ with the element in $L(u)- \{a\}$ to obtain an equitable, arborable $L$-coloring of $G$.

Now, let $\mathcal{B}$ be the set of two element subsets of $B$.  Let $F: A \rightarrow \mathcal{B}$ be a function that maps each $u \in A$ to an $\{x,y\} \in \mathcal{B}$ so that $x,y \in N_G(u)$, $f(x) \in L(u) - \{a\}$, $f(x)=f(y)$, and there is a path in $G[f^{-1}(f(x))]$ connecting $x$ and $y$ (in the case multiple elements of $\mathcal{B}$ satisfy these conditions one of those elements is arbitrarily chosen for $F(u)$).  Now, for each $z \in B$, let $m_z = \sum_{P \in \mathcal{B}, z \in P} |F^{-1}(P)|$.  Then, note that
$$\sum_{z \in B} m_z = 2 \sum_{P \in \mathcal{B}} |F^{-1}(P)| = 2|A|= 2(l+1).$$
It is clear that for each $z \in B$ with $m_z \geq 1$, $d_{G_2}(z) \geq 1$.  We claim that for each $z \in B$, $m_z \leq 3$.  To see why this is so, suppose that there is a $t \in B$ with $m_t \geq 4$.  Note that $m_t \geq 4$ implies that $t$ is adjacent to at least 4 vertices in $A$.  This along with the fact that $d_{G_2}(t) \geq 1$ implies $d_G(t) \geq 5 > \Delta(G)$ which is a contradiction.

So, for each $z \in B$, $m_z \leq 3$.  Since $|B| = l-1$, there must be four distinct vertices $t_1,t_2,t_3,t_4 \in B$ such that $m_{t_i}=3$ for each $i \in [4]$.  This however implies that $d_G(t_i)=4$ for each $i \in [4]$ which contradicts the fact that $G$ has at most three vertices of degree~4.
\\
\\

We now turn our attention to Statement~(ii).  For the sake of contradiction, suppose that $G$ is not equitably 2-list arborable.  Let $L$ be a 2-assignment for $G$ for which there is no equitable, arborable $L$-coloring of $G$.  Since $G$ is not a complete graph, by Theorem~\ref{thm: upperarbor}, we know that there is an arborable $L$-coloring of $G$.  Since there is no equitable, arborable $L$-coloring of $G$, each arborable $L$-coloring of $G$ must have exactly one color class with size larger than $\lceil n/2 \rceil$.  Among all arborable $L$-colorings of $G$, choose one, $f$, so that the largest color class associated with $f$ is as small as possible.

Suppose $a$ is the color in $\bigcup_{v \in V(G)} L(v)$ for which $|f^{-1}(a)| > \lceil n/2 \rceil$.  We let $m=|f^{-1}(a)|$, $A = f^{-1}(a)$, and $B = V(G) - A$.  Note that $|B| = n-m$ and no color other than $a$ is used by $f$ more than $\lceil n/2 \rceil - 1$ times.  We also let $G_1$ be the forest $G[A]$ and $G_2 = G[B]$.  Notice that for each $u \in A$, $u$ is adjacent to two vertices $x, y \in B$ (in $G$) such that: $f(x) \in L(u) - \{a\}$, $f(x)=f(y)$, and there is a path in $G[f^{-1}(f(x))]$ connecting $x$ and $y$.  This is because if this was not the case, we could recolor $u$ with the element in $L(u)- \{a\}$ to obtain an arborable $L$-coloring of $G$ with every color class of size less than $m$.

Now, let $\mathcal{B}$ be the set of two element subsets of $B$.  Let $F: A \rightarrow \mathcal{B}$ be a function that maps each $u \in A$ to an $\{x,y\} \in \mathcal{B}$ so that $x,y \in N_G(u)$, $f(x) \in L(u) - \{a\}$, $f(x)=f(y)$, and there is a path in $G[f^{-1}(f(x))]$ connecting $x$ and $y$ (in the case multiple elements of $\mathcal{B}$ satisfy these conditions one of those elements is arbitrarily chosen for $F(u)$).  Now, for each $z \in B$, let $m_z = \sum_{P \in \mathcal{B}, z \in P} |F^{-1}(P)|$.  Then, note that
$$\sum_{z \in B} m_z = 2 \sum_{P \in \mathcal{B}} |F^{-1}(P)| = 2|A|= 2m.$$
Similar to the proof of Statement~(i), it is clear that for each $z \in B$ with $m_z \geq 1$, $d_{G_2}(z) \geq 1$, and for each $z \in B$, $m_z \leq 3$.

Now, let $T = \{z \in B : m_z \geq 1 \}$.  Since $\sum_{z \in B} m_z = 2m$ and each term in the sum $\sum_{z \in B} m_z$ is at most 3, we have that $|T| \geq \lceil 2m/3 \rceil$.  Now, notice that $|E(G)| = |E(G_1)|+|E(G_2)| + |E_G(A,B)|$.  Using this equation as our starting point, we see that:
\begin{align*}
|E(G_1)| = 2n - |E(G_2)| - |E_G(A,B)| &= 2n - \frac{\sum_{z \in B} d_{G_2}(z)}{2} - \sum_{z \in B} (4-d_{G_2}(z)) \\
&= 2n - 4|B| + \frac{\sum_{z \in B} d_{G_2}(z)}{2} \\
&= 2n - 4(n-m) + \frac{\sum_{z \in B} d_{G_2}(z)}{2} \\
& \geq 4m - 2n + \frac{\sum_{z \in T} d_{G_2}(z)}{2} \\
& \geq 4m - 2n + \frac{|T|}{2} \\
& \geq 4m - 2n + \frac{ \lceil 2m/3 \rceil}{2} .
\end{align*}
Now, we claim that $m \leq 4m - 2n + \frac{ \lceil 2m/3 \rceil}{2}$.  To why this is so, note that $m \geq \lceil n/2 \rceil + 1$.  Then, when $n \in \{6,8,10\}$, it is easy to see that $2n \leq 3(n/2 + 1) + \frac{ \lceil (n + 2)/3 \rceil}{2} \leq 3m + \frac{ \lceil 2m/3 \rceil}{2},$ and when $n \in \{7, 9, 11, 13, 15\}$, it is easy to see that $2n \leq 3(n+3)/2 + \frac{ \lceil (n + 3)/3 \rceil}{2} \leq 3m + \frac{ \lceil 2m/3 \rceil}{2}$.

So, we have that $|E(G_1)| \geq m$.  Since $G_1$ is a graph on $m$ vertices, we have that $G_1$ contains a cycle which implies that $f$ is not an arborable $L$-coloring of $G$.  This however is a contradiction.
\end{proof}

It is worth mentioning that in some sense Statement~(i) of Theorem~\ref{thm: four} is best possible since $K_5$ is not equitably 2-list arborable, but $K_5$ minus an edge is equitably 2-list arborable.  Also, in light of Conjecture~\ref{conj: stronger}, we suspect that all connected, 4-regular graphs with the exception of $K_5$ are equitably 2-list arborable.  So, we expect that Statement~(ii) can be improved quite a bit.

\section{A General Tool and its Applications} \label{general}

We will now prove a generalization of Lemma~\ref{lem: pelsmajerlike} which was used in the previous section and proven by Zhang in 2016.\footnote{Lemma~\ref{lem: pelsmajerlike} is similar in flavor to Lemma~3.1 in~\cite{KP03} which is a well-known Lemma that, along with its generalizations, has been used by many researchers to prove results about equitable choosability.}

First, we need some terminology. Suppose $G$ is a graph and $S \subseteq V(G)$.  Suppose $L$ is a $k$-assignment for $G$, and suppose $L'$ is the $k$-assignment for $G-S$ obtained by restricting the domain of $L$ to $V(G-S)$.  Suppose that $f$ is an equitable, arborable $L'$-coloring of $G-S$.  Now, for each $v \in S$, let
$$D(v) = L(v) - \{ c \in L(v) : |f^{-1}(c) \cap N_G(v)| \geq 2 \}.$$
Moreover, for each $v \in S$, let $D_d(v) = \{ c \in L(v) : |f^{-1}(c) \cap N_G(v)| = 1 \}$ and $D_s(v) = \{ c \in L(v) : |f^{-1}(c) \cap N_G(v)| = 0 \}$.  Clearly, $\{D_d(v), D_s(v) \}$ is a partition of $D(v)$.  If $c \in D_d(v)$ we say that color $c$ is \emph{dangerous with respect to $v$}, and if $c \in D_s(v)$ we say that color $c$ is \emph{safe with respect to $v$}.\footnote{From this point onward, we use these names for the lists whenever we are deleting a set $S$ of vertices from a graph $G$ in hopes of extending an equitable, arborable list coloring of $G-S$ to an equitable, arborable list coloring of $G$.}  Using this notation, the following observation is immediate.

\begin{obs} \label{obs: obvious}
If $v \in S$ and $|N_G(v) - S| = t$, then $|D(v)| \geq k - \lfloor t/2 \rfloor$ and $|D_s(v)| \geq k - t$.
\end{obs}

We are now ready to prove Lemma~\ref{lem: general}, which we restate.

\begin{customlem} {\bf \ref{lem: general}}
Suppose $m \in \N$ and $S= \{x_1, \ldots, x_{mk} \}$ where $x_1, \ldots, x_{mk}$ are distinct vertices of $G$.  Suppose that $L$ is a $k$-assignment for $G$, and $L'$ is the $k$-assignment for $G-S$ obtained by restricting the domain of $L$ to $V(G-S)$. Let $f$ be an equitable, arborable $L'$-coloring of $G-S$.

Suppose that there is an arborable $D$-coloring $g$ of $G[S]$ such that: $g$ uses no color more than $m$ times and for each $c \in g(S)$ there is at most one vertex $v \in g^{-1}(c)$ with the property that $c$ is dangerous with respect to $v$.  Then, the function $h : V(G) \rightarrow \bigcup_{v \in V(G)} L(v)$ given by
\[ h(v) = \begin{cases}
f(v) & \text{if $v \notin S$} \\
g(v) & \text{if $v \in S$}
\end{cases}
\]
is an equitable, arborable $L$-coloring of $G$.	
\end{customlem}

\begin{proof}
Clearly, $h$ is an $L$-coloring of $G$.  We will first show that $h$ is an arborable $L$-coloring of $G$; that is, we will show that for each color $c$ in the range of $h$, $G[h^{-1}(c)]$ is acyclic.  Note that $h^{-1}(c) = f^{-1}(c) \cup g^{-1}(c)$.  We know that $G[f^{-1}(c)]$ and $G[g^{-1}(c)]$ are vertex disjoint, and each of these graphs is acyclic.  Since there is at most one vertex $v \in g^{-1}(c)$ with the property that $c$ is dangerous with respect to $v$, we know that there is at most one vertex in $g^{-1}(c)$ that is adjacent in $G$ to a vertx in $f^{-1}(c)$.  Moreover, such a vertex can only be adjacent to exactly one vertex in $f^{-1}(c)$.  Thus, $G[h^{-1}(c)]$ is either $G[f^{-1}(c)]+G[g^{-1}(c)]$ or $G[f^{-1}(c)]+G[g^{-1}(c)]$ with a single edge added between one vertex of $G[f^{-1}(c)]$ and one vertex of $G[g^{-1}(c)]$.  Since adding a single edge between two vertex disjoint forests cannot create a cycle, it follows that $G[h^{-1}(c)]$ is acyclic.

Finally, to see that $h$ is an equitable, arborable $L$-coloring of $G$, notice that since $f$ uses no color more than $\lceil |V(G-S)|/k \rceil = \lceil |V(G)|/k \rceil - m$ times and $g$ uses no color more than $m$ times, $h$ uses no color more than $\lceil |V(G)|/k \rceil$ times.
\end{proof}

We now use Lemma~\ref{lem: general} to improve upon Theorem~\ref{pro: pathpower} when $p \geq 3$.

\begin{custompro} {\bf \ref{pro: p-1}}
Suppose $n,p \in \N$, $p \geq 3$, and $G=P_n^p$.  Then, $G$ is equitably $(p-1)$-list arborable.
\end{custompro}

\begin{proof}
Suppose that $p$ is a fixed natural number satisfying $p \geq 3$.  We will prove the desired result by induction on $n$.  The result is obvious when $n \leq 2p-2$.  So, suppose that $n \geq 2p-1$ and that the desired statement holds true for all natural numbers less than $n$.

Suppose that the vertices of the underlying copy of $P_n$ used to form $G$ (in order) are: $v_1, \ldots, v_n$.  Suppose also that $L$ is an arbitrary $(p-1)$-assignment for $G$, and $S= \{v_1, \ldots, v_{2p-2} \}$.  Then, the inductive hypothesis tells us that an equitable, arborable $L'$-coloring $f$ of $G-S$ exists where $L'$ is the $k$-assignment for $G-S$ obtained by restricting the domain of $L$ to $V(G-S)$.  In order to show that there is an equitable, arborable $L$-coloring of $G$, we will construct an arborable $D$-coloring of $G[S]$ that satisfies the hypotheses of Lemma~\ref{lem: general}.

We begin with several observations.  First, for each $i \in [p-2]$, $D_s(v_i)=D(v_i)=L(v_i)$ which implies $|D_s(v_i)|=p-1$.  Second, since $v_{p-1}$ has one neighbor in $G-S$, we know that $|D(v_{p-1})|=p-1$ and $|D_s(v_{p-1})| \geq p-2$.  Also, for each $i \in \{p+1, \ldots, 2p-2 \}$, since $v_i$ has at most $i-p+2$ neighbors in $G-S$,
$$|D(v_i)| \geq (p-1) - \left \lfloor \frac{i-p+2}{2} \right \rfloor.$$
From this inequality it is easy to verify that $|D(v_i)| \geq 2p-1-i$ for each $i \in \{p+1, \ldots, 2p-2 \}$.  Finally, one of the following statements must be true since $v_p$ has at most two neighbors in $G-S$: (1) $|D(v_p)|=p-1$ or (2) $|D(v_p)|=p-2$.  We will construct an arborable $D$-coloring of $G[S]$ that satisfies the hypotheses of Lemma~\ref{lem: general} in each of these cases.

In case (1) begin by greedily coloring the vertices $v_{2p-2}, \ldots, v_{p}$  with $p-1$ pairwise distinct colors from $D(v_{2p-2}), \ldots, D(v_p)$ respectively (this is possible since in case (1) we have that $|D(v_i)| \geq 2p-1-i$ for each $i \in \{p, \ldots, 2p-2 \}$).  Note that $v_i$ may be colored with a dangerous color with respect to $v_i$ for each $i \in \{p, \ldots, 2p-2 \}$.  Then, greedily color $v_{p-1}, \ldots, v_1$ with $p-1$ pairwise distinct colors from $D_s(v_{p-1}), \ldots, D_s(v_1)$ respectively.  Call the resulting coloring of $G[S]$, $g$.  It is easy to see that $g$ uses no color more than two times and is therefore an arborable $D$-coloring of $G[S]$.  Furthermore, by construction, for each color class $C$ of $g$ there is at most one vertex $v \in C$ that is colored with a color in $D_d(v)$.  Lemma~\ref{lem: general} immediately implies that there is an equitable, arborable $L$-coloring of $G$.

In case (2) begin by greedily coloring the vertices $v_{2p-2}, \ldots, v_{p+1}, v_{p-1}$  with $p-1$ pairwise distinct colors from $D(v_{2p-2}), \ldots, D(v_{p+1}), D(v_{p-1})$ respectively (this is possible since we have that $|D(v_{p-1})|=p-1$ and $|D(v_i)| \geq 2p-1-i$ for each $i \in \{p+1, \ldots, 2p-2 \}$).  Note that $v_i$ may be colored with a dangerous color with respect to $v_i$ for each $i \in \{p-1, p+1, \ldots, 2p-2 \}$.  Now, notice that in case (2) we must have that $D(v_p)=D_s(v_p)$.  So, we greedily color $v_{p}, v_{p-2}, \ldots, v_1$ with $p-1$ pairwise distinct colors from $D_s(v_p), D_s(v_{p-2}), \ldots, D_s(v_1)$ respectively (this is possible since $|D_s(v_p)| \geq p-2 \geq 1$).  Call the resulting coloring of $G[S]$, $g$.  It is easy to see that $g$ uses no color more than two times and is therefore an arborable $D$-coloring of $G[S]$.  Furthermore, by construction, for each color class $C$ of $g$ there is at most one vertex $v \in C$ that is colored with a color in $D_d(v)$.  Lemma~\ref{lem: general} immediately implies that there is an equitable, arborable $L$-coloring of $G$.
\end{proof}

Finally, we use Lemma~\ref{lem: general} to prove Conjecture~\ref{conj: HSanalogue} for powers of cycles.

\begin{customthm} {\bf \ref{pro: cyclepower}}
Suppose that $p \geq 2$ and $n \geq 2p+2$.  If $G = C_{n}^p$, then $G$ is equitably $k$-list arborable for each $k$ satisfying $k \geq p+1$.
\end{customthm}

\begin{proof}
Suppose $p,k$ are fixed natural numbers satisfying $p \geq 2$ and $k \geq p+1$.  We will prove the result by induction on $n$.  Notice that when $n$ satisfies $2p+2 \leq n \leq 2k$ the desired result is obvious.

So, we may assume that $n > 2k$. Suppose that the vertices of the underlying copy of $C_n$ used to form $G$ in cyclic order are: $v_1, \ldots, v_n$.  Suppose that $L$ is an arbitrary $k$-assignment for $G$, and let $S = \{v_1, \ldots, v_{2k} \}$.  Notice that $G-S$ is a copy of $P_{n-2k}^p$.  So, Theorem~\ref{pro: pathpower} tells us there is an equitable, arborable $L'$-coloring $f$ of $G-S$ where $L'$ is the $k$-assignment for $G-S$ obtained by restricting the domain of $L$ to $V(G-S)$.  In order to show that there is an equitable, arborable $L$-coloring of $G$, we will construct an arborable $D$-coloring of $G[S]$ that satisfies the hypotheses of Lemma~\ref{lem: general}.

It is easy to see that for each $i \in [k]$, $|N_G(v_i) - S| \leq \max\{p+1-i, 0 \}$ and $|N_G(v_{k+i}) - S| \leq \max \{p+i-k,0 \}$.  This implies that for each $i \in [k]$, $|D_s(v_i)| \geq k - \max\{p+1-i, 0 \} \geq i$.  So, we can greedily color $v_{1}, \ldots, v_k$ with $k$ pairwise distinct colors from $D_s(v_1), \ldots, D_s(v_k)$ respectively.  Similarly, for each $i \in [k]$, $|D_s(v_{k+i})| \geq k - \max\{p+i-k, 0 \} \geq k-i+1$.  Consequently, we can greedily color $v_{2k}, \ldots, v_{k+1}$ with $k$ pairwise distinct colors from $D_s(v_{2k}), \ldots, D_s(v_{k+1})$ respectively.  Call the resulting coloring of $G[S]$, $g$.

It is easy to see that $g$ uses no color more than two times and is therefore an arborable $D$-coloring of $G[S]$.  Furthermore, by construction, for each $v \in S$, $g(v)$ is safe with respect to $v$.  Lemma~\ref{lem: general} immediately implies that there is an equitable, arborable $L$-coloring of $G$.
\end{proof}

{\bf Acknowledgment.}  The authors would like to thank the anonymous referees for their helpful comments that improved the readability of this paper.


\begin{thebibliography}{99}
{\small

\bibitem{B64} L. W. Beineke, Decompositions of complete graphs into forests, \emph{Magyar Tud. Akad. Mat. Kutat{\'o} Int. K{\"o}zl.} 9 (1964), 589-594.

\bibitem{BK00} O. V. Borodin, A. V. Kostochka, B. Toft, Variable degeneracy: Extensions of Brooks' and Gallai's theorems, \emph{Discrete Mathematics} 214 (2000), 101-112

\bibitem{BD95} M. Borowiecki, E. Dragas-Burchardt, P. Mih{\'o}k, Generalized list colouring of graphs, \emph{Discussiones Mathematicae Graph Theory} 15 (1995), 185-193.

\bibitem{CK69} G. Chartrand, H. V. Kronk, The point-arboricity of planar graphs, \emph{J. London Math Soc.} 44 (1969), 612-616.

\bibitem{CL94} B. -L. Chen, K. -W. Lih, P. -L. Wu, Equitable coloring and the maximum degree, \emph{Eur. J. Combin.} 15 (1994), 443-447.

\bibitem{DD18} E. Drgas-Burchardt, J. Dybizba{\'n}ski, H. Furma{\' n}czyk, E. Sidorowicz, Equitable list vertex colorability and arboricity of grids, \emph{Filomat} 32 (18) (2018), 6353-6374.

\bibitem{DF20} E. Drgas-Burchardt, H. Furma{\' n}czyk, E. Sidorowicz, Equitable improper choosability of graphs, \emph{Theoretical Computer Science} 844 (6) (2020), 35-45.

\bibitem{DW18} A. Dong, J. Wu, Equitable coloring and equitable choosability of planar graphs without chordal 4- and 6-cycles, \emph{Discrete Mathematics \& Theoretical Computer Science} 21 (3) (2019), 1-21.

\bibitem{DZ18} A. Dong, X. Zhang, Equitable coloring and equitable choosability of graphs with small maximum average degree, \emph{Discussiones Mathematicae Graph Theory} 38 (2018), 829-839.

\bibitem{E64} P. Erd\H{o}s, Problem 9, In: M. Fiedler, editor, \emph{Theory of Graphs and Its Applications}, Proc. Sympos., Smolenice, 1963, Publ. House Czechoslovak Acad. Sci. Prague, 1964, 159.]

\bibitem{ET79} P. Erd\H{o}s, A. L. Rubin, H. Taylor, Choosability in graphs, \emph{Congressus Numerantium} 26 (1979), 125-127.

\bibitem{HS70} A. Hajn\'{a}l, E. Szemer\'{e}di, Proof of a conjecture of Erd\H{o}s, In: A R\'{e}nyi, V. T. S\'{o}s, editors, \emph{Combinatorial Theory and Its Applications}, Vol. II, North-Holland, Amsterdam, 1970, 601-623.

\bibitem{KM18} H. Kaul, J. A. Mudrock, M. J. Pelsmajer, Total equitable list coloring, \emph{Graphs and Combinatorics} 34 (2018), 1637-1649.

\bibitem{KK13} H. A. Kierstead, A. V. Kostochka, Equitable list coloring of graphs with bounded degree, \emph{J. of Graph Theory} 74 (2013), 309-334.

\bibitem{KP03} A. V. Kostochka, M. J. Pelsmajer, D. B. West, A list analogue of equitable coloring, \emph{J. of Graph Theory} 44 (2003), 166-177.

\bibitem{LB09} Q. Li, Y. Bu, Equitable list coloring of planar graphs without 4- and 6-cycles, \emph{Discrete Mathematics} 309 (2009), 280-287.

\bibitem{LZ21} Y. Li, X. Zhang, Equitable list tree-coloring of bounded treewidth graphs, \emph{Theoretical Computer Science} 855 (6) (2021), 61-67.

\bibitem{L98} K. -W. Lih, The equitable coloring of graphs, In: D. -Z. Du, P. Pardalos, editors. \emph{Handbook of Combinatorial Optimization}, Vol. III, Kluwer, Dordrecht, 1998, 543-566.

\bibitem{LW96} K. -W. Lih, P. -L. Wu, On equitable coloring of bipartite graphs, \emph{Discrete Mathematics} 151 (1996), 155-160.

\bibitem{LC12} W. Lin, G. Chang, Equitable colorings of Cartesian products of graphs, \emph{Discrete Applied Mathematics} 160 (2012), 239-247.

\bibitem{M73} W. Meyer, Equitable Coloring, \emph{Amer. Math. Monthly} 80 (1973), 920-922.

\bibitem{M18} J. Mudrock, On the list coloring problem and its equitable variants, Ph.D. Thesis, Illinois Institute of Technology, 2018.

\bibitem{MC19} J. Mudrock, M. Chase, I. Kadera, T. Wagstrom, A note on the equitable choosability of complete bipartite graphs, \emph{to appear in Discussiones Mathematicae Graph Theory}.

\bibitem{MM19} J. Mudrock, M. Marsh, T. Wagstrom, On list equitable total colorings of the generalized theta graph, \emph{to appear in Discussiones Mathematicae Graph Theory}.

\bibitem{NZ20} B. Niu, X. Zhang, Y. Gao, Equitable partition of plane graphs with independent crossings into induced forests, \emph{Discrete Mathematics} 343 (5) (2020), 111792.

\bibitem{V76} V. G. Vizing, Coloring the vertices of a graph in prescribed colors, \emph{Diskret. Analiz.} no. 29, \emph{Metody Diskret. Anal. v Teorii Kodovi Skhem} 101 (1976), 3-10.

\bibitem{W01} D. B. West, (2001) \emph{Introduction to Graph Theory}.  Upper Saddle River, NJ: Prentice Hall.

\bibitem{WZ13} J- -L. Wu, X. Zhang, H, Li, Equitable vertex arboricity of graphs, \emph{Discrete Mathematics} 313(23) (2013), 2696-2701.

\bibitem{YZ97} H. P. Yap, Y. Zhang, The equitable $\Delta$-coloring conjecture holds for outerplanar graphs, \emph{Bull. Inst. Acad. Sinica} 25 (1997), 143-149.

\bibitem{Z16} X. Zhang, Equitable list point arboricity of graphs, \emph{Filomat} 30(2) (2016), 373-378.

\bibitem{Z162} X. Zhang, Equitable vertex arboricity of graphs, \emph{Discrete Mathematics} 339(6) (2016), 1724-1726.

\bibitem{ZN20} X. Zhang, B. Niu, Equitable partition of graphs into induced linear forests, \emph{Journal of Combinatorial Optimization}, 39(2) (2020), 581-588.

\bibitem{ZN19} X. Zhang, B. Niu, Y. Li, B. Li, Equitable vertex arboricity conjecture holds for graphs with low degeneracy, arXiv: 1908.05066v3 (preprint), 2019.

\bibitem{ZW14} X. Zhang, J. -L. Wu, A conjecture on equitable vertex arboricity of graphs, \emph{Filomat} 28(1) (2014), 217-219.

\bibitem{ZW11} X. Zhang, J. -L. Wu, On equitable and equitable list colorings of series-parallel graphs, \emph{Discrete Mathematics} 311 (2011), 800-803.

\bibitem{ZZ21} H. Zhang, X. Zhang, Theoretical aspects of equitable partition of networks into sparse modules, \emph{Theoretical Computer Science}, 871 (2021), 51-61.

\bibitem{ZB08} J. Zhu, Y. Bu, Equitable list coloring of planar graphs without short cycles, \emph{Theoretical Computer Science} 407 (2008), 21-28.

\bibitem{ZB10} J. Zhu, Y. Bu, Equitable and equitable list colorings of graphs, \emph{Theoretical Computer Science} 411 (2010), 3873-3876.

\bibitem{ZB15} J. Zhu, Y. Bu, X. Min, Equitable list-coloring for $C_5$-free plane graphs without adjacent triangles, \emph{Graphs and Combinatorics}, 31 (2015), 795-804.

}

\end{thebibliography}
\end{document}